\documentclass[12pt]{article}
\usepackage{amssymb,amsmath,amscd,amsthm,amsfonts}
\usepackage[numeric]{amsrefs}
\usepackage{algorithmic}
\usepackage{graphicx}%
\usepackage{amsmath}%
\usepackage{amsfonts}%
\usepackage{amssymb}
\providecommand{\U}[1]{\protect\rule{.1in}{.1in}}
\newtheorem{theorem}{Theorem}[section]

\newtheorem{proposition}[theorem]{Proposition}
\theoremstyle{definition}
\newtheorem{definition}[theorem]{Definition}

\renewcommand\a{\alpha}
\renewcommand\b{\beta}
\newcommand\g{\gamma}
\renewcommand\d{\delta}
\newcommand\e{\varepsilon}
\renewcommand\l{\lambda}

\newcommand\G{\Gamma}
\newcommand\f{\frac}

\newcommand{\N}{{\mathbb{N}}}

\newcommand{\Z}{{\mathbb{Z}}}
\newcommand{\Int}{{\mathbb{Z}}}
\newcommand{\R}{{\mathbb{R}}}

\newcommand{\RP}{{\mathbb{RP}}}

\newcommand{\E}{{\mathbb{E}}}

\renewcommand\i{^{-1}}
\renewcommand\({\left(}
\renewcommand\){\right)}

\numberwithin{equation}{section}

\begin{document}

\title{Non-Abelian Analogs of Lattice Rounding}

\author{Evgeni Begelfor\\Department of Computer Science\\
The Hebrew University of Jerusalem\\   \texttt{begelfor@gmail.com}
\and Stephen D. Miller\thanks{Supported by NSF grant DMS-1201362.}\\Department of Mathematics\\
Rutgers University\\
\texttt{miller@math.rutgers.edu}
\and Ramarathnam Venkatesan\\ Microsoft
Research
\\ Redmond, WA and Bangalore, India
\\ \texttt{venkie@microsoft.com} }
\date{January 12, 2015}
\maketitle

\begin{abstract}
Lattice rounding in Euclidean space can be viewed as finding the nearest point in the orbit of an action by a discrete group, relative to the norm inherited from the ambient space.  Using this point of view, we initiate the study of non-abelian analogs of lattice rounding involving matrix groups.  In one direction, we give an algorithm for solving a normed word problem when the inputs are random products over a basis set, and give theoretical justification for its success.  In another direction, we prove a general inapproximability result which essentially rules out {\em strong approximation algorithms} (i.e., whose approximation factors depend only on dimension) analogous to  LLL in the general case.
\end{abstract}


\section{Introduction}

Given a basis $\{a_i\}_{i=1}^n$ of a lattice $L\subset\R^n$ and a vector $y\in\R^n$,  the Lattice Rounding Problem (\textsc{lrp}) in Euclidean space asks to  find
 $\underset{z\in L}{\arg\min}$ $||z-y||_{2}$, that is,
 a vector $z\in L$ nearest to $y$.
This problem is
very closely related to the lattice basis reduction problem of finding a good basis
for $L$, which informally is to find another basis $\{b_{i}\}_{i=1}^{n}$ for
$L$ whose elements are as orthogonal as possible. The motivation is that given
such a good basis $\{b_{i}\}_{i=1}^{n}$, \textsc{lrp} may be easy. To wit, if
$L=\Z
^{n}$  a good basis is trivial to find, and \textsc{lrp} can be solved by
coordinate-wise rounding. For general $L$ and bases $\{a_{i}\}_{i=1}^{n},$
one has NP-hardness results for exact and approximate versions of \textsc{lrp}
\cite{arora93hardness,dinur98approximatingcvp}, and their study is an active
area of research.

The presumed hardness of these problems also has led to constructions of
cryptosystems.  This typically involves three main ingredients:
\begin{enumerate}
  \item[(a)] \textbf{Good Basis.} Generation of a basis  $\{b_{i}\}_{i=1}^{n}$ for $L$
that is good in the sense that  \textsc{lrp} is easy relative to it on inputs randomly
chosen from some distribution $\nu$.
  \item[(b)] \textbf{Bad Basis.}  Generation of a suitable matrix $M\in SL_n(\Z)$ such that \textsc{lrp} with respect to  $\nu$ is hard relative to the basis $\{a_{i}\}_{i=1}^{n}$, where $a_i=Mb_i$.
  \item[(c)]  \textbf{Public Key System.} One keeps the good basis as the private key
and the bad basis as a public key, and designs an encryption or signature scheme
such that an attack on it would entail solving \textsc{lrp} relative to a bad basis.
\end{enumerate}

\noindent
This paper presents a non-abelian generalization of lattice rounding, and some steps in the direction of ingredients (a) and (b).  Our generalization starts with the viewpoint of $\R^n$ as an additive abelian group and $L$ as a discrete subgroup:~\textsc{lrp} is equivalent to finding the nearest point to $z$ (in the ambient metric) to the orbit of the origin under the action of $L$.  This viewpoint can be extended to a larger class of groups, and spaces upon which they act.  For example, one could consider a Lie group such as  the $n\times n$ invertible matrices $G=GL_n(\R)$, and a discrete subgroup $\G$; this direction quickly leads to  rich mathematical theory connected with dynamics and automorphic forms.  In this case one could choose ambient metrics on $G$ related to
a variety of matrix norms.

Another direction is to consider the action of $G$ on some space $X$ endowed with its own metric.  
For example, $G=GL_n(\R)$ acts on the vector space $X=\R^n$ or even the projective space $\RP^{n-1}$ by the usual multiplication of vectors by matrices.  Let $\G$ as before denote a subgroup of $G$.  A non-abelian analog of lattice rounding   asks to find the closest point in the $\G$-orbit of a fixed vector in $\R^n$, where the closeness is measured using some
natural metric on vectors (but not on matrices, although we do make a restriction on word length for practical reasons).   

Alternatively, if $\G$ and $X$ are themselves endowed with a discrete structure (e.g., $\G$ consists of integral matrices and $X$ consists of integral vectors), we can instead study the problem of recognizing elements of a $\G$-orbit.  To address items (a) and (b) above, is natural to
ask if one can develop analogous positive algorithms for rounding
with good bases and, conversely, negative results for general subgroups $\Gamma$ in
$GL_n(\R)$.  One naive approach would be to modify a generating set $\{g_1,\ldots,g_r\}$ by successively replacing a generator $g_i$ by $g_i g_j^c$, where $j\neq i$ and $c\in \Z$.  In the abelian case such repeated modifications generate any change of lattice basis.  However, in the non-abelian case there are some geometric constraints (such as course quasi-isometry) which may at times dull the effects of such a change.  We do not investigate this direction here.

In Section 3 we consider the  Word Problem on Vectors (\ref{WordProblemonVectors}), for which we propose the Norm Reduction Algorithm (\ref{NRA}).
 The analysis of the latter leads to well-studied  mathematical and
algorithmic topics.  For example, multiplying random elements of $\G$ times a fixed vector can be viewed as a {\em generalized
Markov chain}  (using more than one matrix); the growing vector norms of these products is itself a  generalization of the law of large numbers  (the case of $n=1$).
Additionally, the conditions for the success of  our Norm Reduction Algorithm depend on an analog of the
spectral or  norm gap in Markov chains:~it requires instead a  gap
between {\em Lyapunov exponents} (see (\ref{lyap})).

\subsection*{Some remarks on our generalization}
The generalization of \textsc{lrp} from lattices   $L$ in  $\R^n$
to finitely-generated subgroups $\Gamma=\langle S \rangle$ in $GL_n(\R)$ is neither unique nor straightforward.  Here we seek to make a distinction between
our norms and the word-length metric, since the latter already appears in the existing literature  in
combinatorial group theory and the study of special groups (e.g.,
braid groups \cite{BraidEqns}) from algorithmic and cryptographic points of
view. We informally outline a few issues that guide our formulation.

\textbf{Full (or at least large) dimensionality}: We would like our discrete subgroups
 to not be contained inside some subgroup of much smaller dimension of the ambient group. In $\R^n$ one typically
assumes  the lattice $L$ has full rank, or least has relatively large rank. Its natural matrix analogue is to
require the {\em Zariski closure} of $\G=\langle S \rangle$ be the full group (or at least
correspond to a subgroup having a significant fraction of the
dimension of the full group).  By definition, this means that the full group is the only group containing $S$ which can be defined as the common zeroes of a set of polynomial equations. This ensures $\G=\langle S \rangle$ is non-abelian in as general way as possible.

For example, if $S$ has only diagonal matrices   it cannot generate any non-abelian group, and its  Zariski closure  is at most an $n$-dimensional subgroup of the $n^2$-dimensional group $GL_n(\R)$. In fact, by considering  commuting diagonal matrices  one can
embed subset-sum type problems and get NP-hardness results.
Note that matrices composed of $2\times 2$ blocks along the diagonal can generate non-abelian groups that essentially describe simultaneous problems in dimension 2; nevertheless, the Post Correspondence Problem can be embedded as a
word problem over   $4\times 4$ matrices with $2\times 2$ blocks, proving the
undecidability of the latter \cite{markov}.  However, certain problems can actually become easier in the non-abelian setting: for example, finding the order of a random element
in $S_n$ is much easier than in $(\Z/p\Z)^*$.

  %We remark that in the proof of our  inapproximability Theorem~\ref{negthm}, the matrices (\ref{gedge}) have Zariski closure containing the upper left $GL_n(\R)$ block
%of $GL_{2n+3}(\R)$ (but not the full group, owing to gadgets needed in
%the NP-hardness reduction), and there the analogue of \textsc{lrp} turns
%out to be hard to approximate to any factor depending only on $n$.

\textbf{Metrics:} The distinction between  the word length metric and ambient matrix norm is discussed in some detail in Section 2 below. The former depends   on the
generating set $S$.  In general these can be very different notions of distance, which makes our study difficult -- yet is key to potential cryptographic applications.  We use the Furstenberg-Kesten theory \cite{furst,FK,lepage} of random matrix products to correlate the two (in a probabilistic sense) in certain situations, which is analogous to the ``good basis'' situation described in (a) above.

\textbf{Finite co-volume and compactness} If $L$ has full rank, then $L\backslash \R^n$ is a compact, finite-volume quotient.  However, neither property necessarily extends to the quotients $\Gamma\backslash G$ in many important examples of $\Gamma$ and $G$. Thus
we do not impose this requirement.  Some further comments are given just below in the beginning of the following section.

\subsection*{Outline of this paper}

Section 2 contains some background about different metrics on Lie groups and their discrete subgroups.  Section 3 introduces the statements of the word problems that motivate our results, as well as the Norm Reduction Algorithm (\ref{NRA}), which is rigorously analyzed in Theorem~\ref{positiveresult}.  The Closest Group Element Problem is also given in section 3, along with the statement of its inapproximability result Theorem~\ref{negthm}.  The analysis of the Norm Reduction Algorithm is performed  in Section 4 using results in dynamical systems. Some experimental results on the algorithm are also presented in Section 4.5.  The   proof of Theorem~\ref{negthm} is given in Section 5; it demonstrates a polynomial time reduction from the Traveling Salesman Problem.

We would like to thank Anthony Bloch, Hillel Furstenberg, Nathan Keller, Peter Sarnak, Adi Shamir, Boaz Tsaban, and Akshay Venkatesh
for their helpful comments.  

\section{Background}

Just as a lattice $L=\langle a_1,\ldots, a_n \rangle$ is additively generated by its basis $\{a_i\}$, the subgroups  $\G=\langle g_1,\ldots,g_k\rangle$ we consider
will be finitely generated.  A crucial difference, however, is that the quotient of $\R^n$ by  $L$ is a compact $n$-dimensional torus with finite volume under the usual Lebesgue measure on $\R^n$ (for example, the quotient $\Z^n\backslash \R^n$).  However, this fails to be true for nice examples such as $GL_n(\Z)\backslash GL_n(\R)$ or even $SL_n(\Z)\backslash SL_n(\R)$, both of which are noncompact under the natural group invariant metric inherited from $G$ (the latter quotient, however, does have finite volume).  The theory and construction of both compact and noncompact discrete subgroups of Lie groups involves numerous beautiful subtleties (see \cite{margulis,zimmer}); we do not restrict ourselves to these objects in this paper.

There are two natural notions of size in $\G$, and by extension to the $\G$-orbit of any basepoint $x\in X$:

\begin{enumerate}
\item \emph{Word length metric}: If $S=\{g_1,\ldots,g_k\}$ is a generating set of $\G$ as above, any element $w\in \G$ can be expressed as a finite word in the alphabet $S\cup S\i$.  There may be many possibilities for such a word, taking into account relations amongst the $g_i$ (including the trivial relation $g_ig_i\i=1)$.  The minimal such length among all such expressions is the {\em word length of $w$ with respect to $S$}. 

  The ability to efficiently compute the word length of $w$ enables one to efficiently write it as a minimal length word, simply by successively checking which of the expressions $g_i^{\pm 1} w$ reduces the word length by one.  Finding the word length depends of course on the generating set $S$, which is analogous to the basis of a lattice.
  In analogy with ingredients (a), (b), and (c) above for Euclidean
lattices, we want the word length to be difficult for typical generating sets $S$ of $\G$, yet at the same time  easy for some ``good bases'' $S$; moreover, we would like to be able to    transform
each ``good base'' into a seemingly bad one.

\item \emph{Inherited metric}: Fundamental to   lattice reduction and
rounding is the notion of  metric on the ambient space.   Natural metrics on $G$ and $X$ therefore can be used to give  generalizations of lattice rounding.  Combining this with word length results in problems such as the following: given $\ell \in \N$, $\Gamma\subset GL_n(\R)$, and vectors $y$ and $z\in \R^n$, find $\g\in \G$ such that $||gy-z||_{2\text{ }}$ is minimized over all $\g\in\Gamma$ with word length at most $\ell$.  Thus the length parameter $\ell$ is used to complement (rather than to duplicate) the ambient metric.
% We do not study the hardness of problems in non-abelian cases  over finite fields,  because of the lack of nontrivial metrics there.
\end{enumerate}

Though we do not present any cryptographic systems here, generalizations of attacks on existing cryptosystems  motivate studying rounding problems in  more general settings than  lattices in $\R^n$ alone.
 With some performance enhancing additions, the lattice reduction algorithm LLL \cite{lll}
has long become a valuable tool in  cryptanalysis
\cite{joux98lattice}, and typically is more effective than
the provable guarantees attached to it indicate alone.  Starting
with the original attack of Shamir \cite{shamir82}, some very
effective attacks have been discovered. The attacks  are often based on the Shortest Vector
Problem in lattices:~given a basis  for $L$,
find a nonzero vector in $L$ with minimal norm.
In polynomial time, the LLL algorithm
 finds a vector within a factor of $2^{n/4}$ of being the shortest, a  {\em strong bound} -- i.e., one which  depends
only on the dimension of the lattice, and not on the sizes of the
entries in the lattice basis themselves.  Babai's rounding algorithm \cite{babai} -- which is
based on LLL -- also has this feature for solving lattice rounding problems in Euclidean space.  The fact that this bound depends only on the dimension is crucial for attacks.

In contrast, we prove in Theorem~\ref{negthm} that the analogous question of rounding products of matrices cannot have a polynomial time strong approximation algorithm\footnote{where the approximating factor is a polynomial time computable function of the dimension.} -- unless P=NP.  This is done by creating a polynomial time reduction to the Traveling Salesman Problem, which has a similar inapproximability result.
Thus a strong approximation algorithm like LLL  for rounding in matrix groups is unlikely to exist.

\section{Some non-abelian problems and an algorithm}\label{sec:wordproblems}

We study  problems that arise out of group actions on normed spaces, where
we are concerned with the action of group elements that have short expressions relative to a given basis or generating set.
 We now proceed
to formally define these problems and state some known results.

We shall work with $GL_d(\R)$, the group of all invertible $d\times d$ real matrices, and often with subsets that have integer entries.  Given $g_1,\ldots,g_k\in GL_d(\R)$, we consider the possible products of these matrices up to a certain length bound, and whether or not they can be recognized as such. The word
problem is the algorithmic task of representing a given matrix in this semigroup as a product of the generators:

\begin{equation}
\label{WordProblem}
\text{
\fbox{\fbox{\begin{minipage}{11.5cm}
\begin{algorithmic}
\STATE {\bf Word Problem}
\STATE {\sc \underline{Input}}:~Matrices $g_1,\ldots,g_k$ and $x\in GL_d(\R)$.
\STATE {\sc \underline{Output}}:~An integer $\ell>0$ and indices $1\le s_1,\ldots,s_\ell\le k$ such that $g_{s_1} g_{s_2}\cdots g_{s_\ell}=x$, if such a solution exists.
\end{algorithmic}
\end{minipage}}}}
\end{equation}
This word problem is known to be unsolvable when $d\ge 4$ \cite{mihail}; however, there is an algorithm for specifically constructed generators when $d=2$ \cite{Gurevich} (the case of $d=3$ is open).
It becomes NP-hard for $d\ge 4$ if we bound the word length $\ell$, as we do for all our problems in the rest of the paper:
\begin{equation}
\label{BoundedWordProblem}
\text{
\fbox{\fbox{\begin{minipage}{11.5cm}
\begin{algorithmic}
\STATE {\bf Bounded Word Problem} \STATE {\sc \underline{Input}}:~An integer $L>0$,
 and matrices $g_1,\ldots,g_k$ and $x\in GL_d(\R)$.
\STATE {\sc \underline{Output}}:~Indices $1\le s_1,\ldots, s_L\le k$ such that $g_{s_1} g_{s_2}\cdots
g_{s_L}=x$, if such a solution exist.
\end{algorithmic} 
\end{minipage}}}}
\end{equation}
This problem can be modified to allow for words of length $\le L$.

We now define another related problem, in which the matrices act on vectors:
\begin{equation}
\label{WordProblemonVectors}
\text{
\fbox{\fbox{\begin{minipage}{11.5cm}
\begin{algorithmic}
\STATE {\bf  Word Problem on Vectors}.
\STATE {\sc \underline{Input}}:~An integer $L>0$, matrices $g_1\ldots,g_k\in GL_d(\R)$ with integer entries, and nonzero vectors $v,w\in \Int^d$.
\STATE {\sc \underline{Output}}:~An integer $\ell\le L$ and indices $1\le s_1,\ldots, s_\ell\le k$  such that $g_{s_1} g_{s_2}\cdots g_{s_\ell}v=w$, if such a solution  exists.
\end{algorithmic}
\end{minipage}}}}
\end{equation}
Typically we are interested in instances where $\ell=L$ and the indices $s_j$ are chosen independently and uniformly at random from the above interval.
Using the ambient norm on Euclidean space, we present the following algorithm for this problem:
\begin{equation}\label{NRA}
\text{
\fbox{\fbox{\begin{minipage}{11.5cm}
 \underline{{\bf Norm Reduction Algorithm:}}
\begin{algorithmic}
\STATE Let $j=0$, and $t$ be a fixed parameter.
\REPEAT
\STATE $j=j+1$
 \STATE $s_j = \underset{i}{\arg\min} \|g_i^{-1}w\|$
\STATE
$w=g_{s_j}^{-1}w$
 \UNTIL $w=v$ or  $j=L-t$.
 \STATE Solve for $s_{L-t+1},\ldots,s_{L}$ by exhaustive search.
\end{algorithmic}
\end{minipage}}}}
\end{equation}

\noindent 
We include the option of exhaustive search   for the final $t$ steps in case the algorithm performs worse on smaller words than on larger ones.  Another possibility is to use a memory-length look-ahead algorithm such as in \cite[\S7]{Thomp1}.
The Norm Reduction  Algorithm is rigorously analyzed in the next section, where it is related to a maximal likelihood algorithm.
Its success depends on some mild yet complicated conditions on generators $g_1,\ldots,g_k$ that come from dynamics.  Theorem~\ref{positiveresult} in the next section gives a rigorous upper bound on the error probability of this algorithm.  We give a successful numerical example in  Table~\ref{resultstable} in Section~\ref{numericalexample}, along with how   Theorem~\ref{positiveresult}'s constants pertain to it.

One can also define a related rounding problem, whose analysis and algorithms are quite similar.  Instead, we will focus   on the following {\em matrix rounding} question: finding a short word in a semigroup  closest to a given one (with an   length constraint imposed for practical reasons).
\begin{equation}
\label{CGEP}
\text{
\fbox{\fbox{\begin{minipage}{11.5cm}
\begin{algorithmic}
\STATE {\bf Closest Group Element Problem}  (\textsc{cgep})
\STATE {\sc \underline{Input}}:~A positive integer $L>0$, and matrices $g_1,\ldots,g_k$ and $z \in GL_d(\R)$.
\STATE {\sc \underline{Output}}:~The closest word of length $\le L$ in the $g_i$ to $z$.
\end{algorithmic}
\end{minipage}}}}
\end{equation}

\noindent Though the problem can be stated for various notions of distance, we will use   the sum-of-squares matrix distance
\begin{equation}\label{l2matnorm}
    \Big\| \,(a_{ij})\,-\,(b_{ij})\,\Big\|^2 \ \ = \ \ \sum_{i,j}|a_{ij}-b_{ij}|^2 
\end{equation}
in studying this problem.

Our main result about the \textsc{cgep} problem is the following negative result, which comes close to ruling out the existence of an algorithm such as LLL that approximates the closest element up to a constant factor depending only on the dimension.  In the following we denote by  $CGEP(g_1,\ldots,g_k,z,L)$ the solution to the \textsc{cgep} problem as above.

\begin{theorem}\label{negthm}
Let $f:\Z_{>0}\rightarrow [1,\infty)$ be a polynomial time computable function.
If there exists a polynomial time algorithm $A$ which, given the input of a \textsc{cgep} problem as in (\ref{CGEP}), always outputs a word $w'$ of length $\le L$ in the $g_i$ such that
\begin{equation}\label{negthm1}
    \|w'\,-\,z\| \ \ \leq  \ \ f(d) \,  \|\,CGEP(g_1,\ldots,g_k,z,L)\,-\,z\,\|
    \, ,
\end{equation}
 then $P=NP$.
\end{theorem}

\noindent
It is an interesting open problem whether or not the approximation factor can instead depend on the sizes of the entries.

\section{Maximum Likelihood Algorithms}\label{sec:positiveresult}

In this section we give and analyze a simple algorithm to solve the Word Problem on Vectors (\ref{WordProblemonVectors}):~try to reduce the norm at each step, or put differently, attempt to use the norm as a proxy for  word length.  This involves studying some background from dynamics related to random products of matrices, first studied by Furstenberg and Kesten \cite{FK,furst}.
  Our results are sensitive  to certain  conditions related to the generators, which we describe before stating our result.  These are discussed thoroughly in the book \cite{lepage}, which serves as a general reference for background material on the topic of this section.  In addition, several of the techniques and arguments in this section are taken from \cite{lepage}.

Let $S=\{g_1,\ldots,g_k\}$ denote a finite subset of $G=GL_d(\R)$, and $T=\langle S \rangle$ the semigroup it generates.  Throughout this section we will use $\|g\|$ to denote the operator norm of a matrix $g$.  We make the following standing assumptions on the set $S$ throughout this section:
\begin{enumerate}
  \item[{\bf A1}.] $T$  is {\em contracting} in the sense of
  \cite[Definition III.1.3]{bougerol}.  This means that $T$ has a sequence of matrices $M_1, M_2,\ldots$ such that $M_n/\|M_n\|$ converges to a rank 1 matrix. It is readily seen (using Jordan canonical form) that this condition holds automatically if $S$ (or even $T$) contains a matrix with an eigenvalue strictly larger than its others in modulus.
  \item[{\bf A2}.] $T$ is {\em strongly irreducible}: there is no finite union of proper vector subspaces of $\R^d$ which is stabilized by each element of $T$.  Equivalently, the same statement holds with $T$ replaced by the {\em group} generated by $S$ (\cite[p.~48]{bougerol}).
  \item[{\bf A3}.] The operator norms $\|g_j^{-1}g_i\|$, $j\neq i$, are all at least some constant $N>1$.
\end{enumerate}

We  prove the following result about the probability of success of the   Norm Reduction Algorithm (\ref{NRA}). This gives a strong indication (along with
numerical testing)  that norm reduction is a suitable algorithm for
solving the Word Problem on Vectors (\ref{WordProblemonVectors}).
  It is also often possible to show that the group generated by $S$ is free by deriving a quantitative version of the well-known Ping-Pong Lemma.  We do not address these issues in this version of the paper.

\begin{theorem}\label{positiveresult}
Let $S=\{g_1,\ldots,g_k\}$ be a fixed subset of $GL_d(\R)$ and $v$ a fixed nonzero vector in $\Z^d$.
Assume properties {\bf A1-3}.  Then there exists  positive quantities $\a$, $B$, and $C$ such that if $h$ is a random product\footnote{I.e., $h=g_{i_1}\cdots g_{i_L}$ where $i_1,\ldots,i_L$ are each chosen independently and uniformly from $\{1,\ldots,k\}$.} of length $L$ elements of $S$, the Norm Reduction Algorithm (\ref{NRA}) recovers $v$ from $hv$ (i.e., solves the Word Problem on Vectors (\ref{WordProblemonVectors})) with probability at least
$$1 \ - \ C\,(L-t)\,(|S|-1)\,N^{-\a},$$ where  $N$ is as defined in assumption {\bf A3} and the parameter $t$ in the algorithm is taken to be at least
$B\log N$. 
\end{theorem}

Roughly speaking, the algorithm succeeds for long enough words when the operator norms $\|g_j\i g_i\|$ are themselves sufficiently large.  Though the constant $N$ is readily computable from the generating set $S$, the numerical values of $C$ and $\a$ are unfortunately more subtle.  We are unable to rigorously prove that $C$ is reasonably small, or that $\a$ is somewhat large.  (It is not clear that these statistics of $\langle S\rangle$ are even  computable in general; see \cite{fuchs,logic,sarnak}.) In particular, one cannot directly take $N\rightarrow\infty$ to get the above error estimate to decay to zero, without possibly simultaneously affecting   $\a$.  However, in concrete examples of generating sets it is possible to make heuristic estimates of the values of $N$ and $\a$ from the proof.  We give such an example in Section~\ref{numericalexample}, in which numerical estimates for these constants give a small error probability in Theorem~\ref{positiveresult}.  Our experiments on this example are vastly better:~the algorithm was successful in nearly all trials we tested for $L\le 1000$ (see Table~\ref{resultstable}).

\subsection{Motivation for the algorithm and its analysis}

Recall the Word Problem on Vectors (\ref{WordProblemonVectors}), in which the matrices in $S$ are assumed to be integral.  One is given $L\in \N$ and  vectors $v$ and $w=hv\in \Z^d$, where $h$ is an unknown word of length at most $L$ in $S$; the problem is to find some word $h'$ of length at most $L$ in $S$ such that $w=h'v$.
%One\footnote{Do we?} can constrain the length to be exactly equal to $L$ by %including the identity matrix in the set $S$; properties {\bf A1} and {\bf A2} %still hold, as does {\bf A3} if all $\|g_j\|$ and $\|g_j\i\| \ge N$.
Were we to have a concrete description of $\nu$ as a product $f\l$, where $f$ is an easily computable function, we could attempt to solve for $h$ using the following maximum likelihood algorithm:
\begin{center}
\fbox{\fbox{\begin{minipage}{12cm}
\begin{algorithmic}
\STATE {\bf Idealized Algorithm:}
\STATE Let $j=0$
\REPEAT
\STATE $j=j+1$
 \STATE $s_j = \underset{i}{\arg\max}
\frac{f(g_i^{-1}w)}{\|g_i^{-1}w\|^d|\det g_i|}$
\STATE
$w=g_{s_j}^{-1}w$
 \UNTIL $w=v$ or $j=L$.
\end{algorithmic}
\end{minipage}}}
\end{center}
Recall the notation $\underset{i}{\arg\max}$ denotes a
value of $i$ which maximizes the expression it precedes.  The
particular expression here represents the change in local density under
the map $w\mapsto g_i\i w$.  The numerator accounts for the
difference between $\nu$ and $\l$, while the denominator represents
the change in the uniform measure $\l$. If successful, the algorithm
produces $h'$ as $g_{s_1}g_{s_2}\cdots$, possibly reconstructing
$h$.   However, it is impractical to assume that   $f$ is easily computable.
Because of this limitation,   we instead use the   simpler, more practical Norm Reduction Algorithm (\ref{NRA}).  It is tantamount to  pretending $f$ equals 1 and that the matrices have determinant
1,  meaning that we seek to minimize $\|g_j^{-1}w\|$ at each stage.

In effect, the Norm Reduction Algorithm (\ref{NRA}) uses the norm as a height function, and proceeds by descent to shorten the word length of $h$ each time.  Of course,  a direct way to measure the word length would be preferable.  The relationship between word length and matrix norm has been studied by several authors, e.g., \cite{gromov,LMR}.

To study the distribution of elements of $T$ and their orbits in
$\R^d$, we need to define some measures.  We let $\mu=\mu_S$ denote
the Dirac measure of $S$ on $G$, meaning that it gives mass
$\f{1}{|S|}$ to each element.
  Given two measures $\mu_1$, $\mu_2$ on $G$, their convolution is defined as the unique measure $\mu_1\ast\mu_2$ satisfying
\begin{equation}\label{mu1convmu2}
    \int_G f(x)\,d\mu_1\!\ast\!\mu_2(x) \ \ = \ \ \int_G \int_G f(x y)\,d\mu_1(x)\,d\mu_2(y) \ \ \ \  \text{for all~}f\in C(G)\,,
\end{equation}
the continuous functions on $G$.  To simplify notation we sometimes  write $\mu_1\ast \mu_2$ simply as $\mu_1\mu_2$; for example
the $n$-fold convolution of $\mu$ with itself will be denoted  as $\mu^n$ (it is the measure giving mass $|S|^{-n}$ to each product of $n$ elements taken from $S$, allowing repetitions).  We can also define the convolution of $\mu$ with any measure $\rho$ on $\RP^{d-1}$: $\mu\ast \rho$ is the unique measure satisfying
\begin{equation}\label{muconvrho}
    \int_{\RP^{d-1}}f(x)\,d\mu\ast\rho(x) \ \ = \ \ \int_G \int_{\RP^{d-1}}f(Mx)\,d\rho(x)\,d\mu(M)\ \ \ \  \text{for all~}f\in C(\RP^{d-1})\,.
\end{equation}
To be concrete, we identify measures on $\RP^{d-1}$ with measures on the unit sphere in $\R^d$ that are invariant under the antipodal map.
Typically the uniform measure $\l$ on $\RP^{d-1}$ is not stabilized by convolution with $\mu$, unless the matrices in $S$ are orthogonal. However, there exist measures $\nu$ on $\RP^{d-1}$ which are $\mu$-invariant:
\begin{equation}\label{muinvnu}
    \mu\ast\nu \ \ = \ \ \nu\,
\end{equation}
(see \cite[Lemma 1.2]{furst}).  Under certain conditions more can be
said about $\nu$, such as its   regularity properties. This measure is not always uniquely determined by $S$, but assumptions {\bf A1} and {\bf
A2} however guarantee the uniqueness of the $\mu$-invariant measure in our setting
(see \cite[Theorem III.4.3.(iii)]{bougerol}).

  The main step in the proof of Theorem~\ref{positiveresult} involves estimating  measures of the subsets of vectors in $\RP^{d-1}$ which get contracted by the operators $g_j\i g_i$.
  Indeed, let $p_j$ equal the probability that the algorithm obtains the wrong value for $g_{s_j}$ at the $j$-th step.  One has that $p_j = \f 1k \sum_{i\le k} p_{ji}$, where $p_{ji}$ is the probability of error in the $j$-th step, conditioned on the correct answer equaling $g_i$.  In terms of the measure $\d_v$, the Dirac measure of $v\in \RP^{d-1}$, this probability can be computed as
\begin{equation}\label{pjiconv}
    p_{ji} \ \ = \ \ \mu^{j-1} \d_v(B_{i}) \, ,
\end{equation}
where $\mu^{j-1}\d_v$ denotes $\mu^{j-1}\ast\d_v$ and
\begin{equation}\label{bi}
\aligned
    B_i \ \ = &  \ \ \{\,  x \in \RP^{d-1} \, \mid \, \exists\, r\neq i \text{~such that~}\|g_r^{-1}g_ix\| < \|x\|  \,          \}\\
    = & \ \  \cup_{r\neq i}B_{r,i}\,,
\endaligned
\end{equation}
with
\begin{equation}\label{bri}
    B_{r,i} \ \ = \ \ \{\,x\in\RP^{d-1} \, \mid \, \|g_r\i g_i x\|<\|x\|\,\}\, .
\end{equation}
Thus the   error probability in Theorem~\ref{positiveresult} is
\begin{equation}\label{errorlastt}
    \operatorname{Prob}_{\text{Error}} \ \ \le \ \ \sum_{j = t+1}^L p_j \ \ = \ \ \f 1k \sum_{\stackrel{\scriptstyle{t< j\le L}}{1\le i\le k}} \mu^{j-1} \d_v(B_{i}) \ \ \le \ \ \f 1k \sum_{\stackrel{\scriptstyle{t< j\le L}}{1\le r\neq i\le k}} \mu^{j-1} \d_v(B_{r,i})  \,.
\end{equation}
The proof therefore amounts to  estimates on $\mu^{j-1} \d_v(B_{r,i})$, which are given in the following subsections.

\subsection{Lyapunov Exponents}

In the remainder of this section, we shall need some technical results and concepts from the literature on random products of matrices.
For the reader's convenience we have chosen to cite background results in the book   \cite{bougerol} wherever possible, while at the same time attempting to correctly attribute the original source of the results.
The top two Lyapunov exponents $\g_1$, $\g_2$ of $S$ are defined through the following limits (see \cite[p.~6]{bougerol}):
\begin{equation}\label{lyap}
    \aligned
       \g_1 \ \ \ = &  \ \ \ \lim_{n\rightarrow\infty} \f 1n\, \E\{\log\|h\| \,\mid\, h\in S^n\}   \, \qquad = & \ \ \f 1n \int_G \log\|M\| \,d\mu^n \qquad  \\
       \g_1+\g_2 \ \  = &  \  \ \ \lim_{n\rightarrow\infty} \f 1n \,\E\{\log\| \wedge^2h\| \,\mid\, h\in S^n\}  \ \  = &  \f 1n \int_G \log\|\wedge^2M\| \,d\mu^n\,,  \\
    \endaligned
\end{equation}
where $\wedge^2 g$ is the operator on $\wedge^2\R^d$ given by $x\wedge y \mapsto gx \wedge gy$ and $\|\cdot\|$ denotes the operator norm (the general Lyapunov exponents are likewise defined inductively through higher exterior powers).  Not only do these limits exist, but in fact a theorem of Furstenberg and Kesten \cite{FK} asserts that the individual terms in the above sets are close to those limits with probability one as $n\rightarrow\infty$.   Under assumptions {\bf A1} and {\bf A2} one has separation between these top two Lyapunov exponents:
\begin{equation}\label{lyapgap}
    \g_1 \ > \ \g_2
\end{equation}
(\cite[Theorem III.6.1]{bougerol}).  We remark that computing or even approximating the Lyapunov exponents is in general difficult \cite{logic}.

We shall use the following variant of (\ref{lyap}), which involves the action of a random product on $\RP^{d-1}$.

\begin{proposition}(Furstenberg \cite{furst}; see \cite[Corollary III.3.4.(iii)]{bougerol}.)
Under  assumption {\bf A2} one has that
\begin{equation}\label{projlyap1}
    \f 1n \,\E\{\log \|h x\|\,\mid\, h\in S^n\} \ \ = \ \ \f 1n \int_G \log\(\f{\|Mx\|}{\|x\|}\)\,d\mu^n \ \ \longrightarrow \ \ \g_1
\end{equation}
uniformly for $x\in \RP^{d-1}$.
\end{proposition}
\noindent Consequently,
\begin{equation}\label{projlyap2}
    \lim_{n\rightarrow\infty}\sup_{x\neq 0}  \, \f 1n \int_G \log\(\f{\|Mx\|}{\|x\|}\)\,d\mu^n \ \ = \ \ \g_1\,.
\end{equation}
Following \cite[p.~55]{bougerol} we use the natural angular distance
\begin{equation}\label{dist}
    \d(x,y) \ \ = \ \ \f{\| x\wedge y\|}{\|x\|\,\|y\|} \ \ = \ \ \sqrt{1-\f{\langle x,y \rangle^2}{\|x\|^2\|y\|^2}}\,,
\end{equation} which is a metric
on $\RP^{d-1}$. It satisfies the following estimate:
\begin{proposition}\label{limsupprop} (See \cite[Proposition III.6.4(ii)]{bougerol}.) For any $x,y\in \RP^{d-1}$,
\begin{equation}\label{limsupassert}
\limsup_{n\rightarrow\infty} \f 1n \int_G \log\(\frac{\d(Mx,My)}{\d(x,y)}\)d\mu^n(M) \ \ \le \ \ \g_2\,-\,\g_1 \ < \ 0 \,.
\end{equation}
\end{proposition}
\begin{proof}
By (\ref{dist})
\begin{equation}\label{lsp1}
\aligned
    \f{\d(Mx,My)}{\d(x,y)} \ \ = &  \ \
\f{\|M(x\wedge y)\|}{\|x\wedge y\|} \f{\|x\|}{\|Mx\|}  \f{\|y\|}{\|My\|}
\\
\le & \ \ \|\wedge^2M\| \f{\|x\|}{\|Mx\|}  \f{\|y\|}{\|My\|}\ , 
\\
\f 1n \log \f{\d(Mx,My)}{\d(x,y)} \ \ \le & \ \ \f 1n \log\|\wedge^2M\| \,-\, \f 1n \log\f{\|Mx\|}{\|x\|} \,-\, \f 1n \log\f{\|My\|}{\|y\|} \,.
\endaligned
\end{equation}
The proposition follows by integrating this inequality over $M$, and appealing to (\ref{lyap}) and (\ref{projlyap1}).
\end{proof}

\subsection{Cocycle integrals}

We have just seen that the integrand
\begin{equation}\label{cocycle1}
    s(M,(x,y)) \ \ = \ \ \log \frac{\d(Mx,My)}{\d(x,y)}
\end{equation}
in (\ref{limsupassert}) tends to be negative on $S^n$.  Our next goal is to show that the integral of an exponential of it is accordingly smaller than 1.
Writing $z$ as shorthand for $(x,y)$, define
\begin{equation}\label{cocycle2}
    {\mathcal S}(n) \ \ = \ \ \sup_z\int_G e^{\a s(M,z)}\,d\mu^n(M)\,,
\end{equation}
which exists for any $\a>0$ since $S$ is finite.  It is proven in \cite[p.~104]{bougerol} that \begin{equation}\label{cocycle3}
    {\mathcal S}(n+m) \ \ \le \ \ {\mathcal S}(n)\,{\mathcal S}(m)\, ,
\end{equation}
using the cocycle identity
\begin{equation}\label{cocycle4}
    s(g_1g_2,v) \ \ = \ \ s(g_1,g_2 v) \ + \ s(g_2,v)
\end{equation}
and a simple change of variables.
According to \cite[Lemma III.5.4]{bougerol}, any matrix $M\in G$ satisfies the inequality
\begin{equation}\label{wedgel}
    \Big| \, \log\|\wedge^2 M\| \,\Big|  \  \ \le  \  \ 2\,\ell(M) \, ,
\end{equation}
where
\begin{equation}\label{ell}
    \ell(M) \ \ = \ \ \max\{\log \|M\|,\, \log \| M\i\|,\, 0\}\,.
\end{equation} 
It follows from (\ref{lsp1}) that
\begin{equation}\label{cocycle5}
    s(M,z) \ \ \le \ \ \log\|\wedge^2 M\| \, + \, 2\log\|M\i\| \ \ \le \ \ 4\,\ell(M)\,.
\end{equation}
If $\ell_{max}$ denotes $\max\{\ell(g) | g\in S\}$, then
 \begin{equation}\label{cocycle6}
    s(M,z) \ \ \le  \ \ 4\,n\,\ell_{max}
 \end{equation}
 on $S^n=$ the support of $\mu^n$, independently of $z$.

\begin{proposition}\label{expprop}(See \cite[Theorem 1]{lepage} and \cite[Proposition V.2.3]{bougerol}.)
For $\a>0$ sufficiently small, there exists $n_0>0$  and $\rho<1$ such that
\begin{equation}\label{ep1}
    \int_G \(\f{\d(Mx,My)}{\d(x,y)}\)^\a d\mu^n(M) \ \ \le \ \ \rho^n
\end{equation}
 for all $x\neq y\in\RP^{d-1}$, and $n\ge n_0$.
\end{proposition}
\begin{proof}
The inequality
$$e^x \ \ \le \ \ 1 \,+\,x\,+\,\f{x^2}{2}e^{|x|}$$
and (\ref{cocycle6}) imply that
\begin{equation}\label{ep2}
    e^{\a s(M,z)} \ \ \le \ \ 1 \,+\, \a\, s(M,z) \,+\, 8\,\a^2 n^2\, \ell_{max}^2\, e^{4n\a \ell_{max}}
\end{equation}
for $M\in S^n$.
Thus the lefthand side of (\ref{ep1}), which is the integral of $e^{\a s(M,z)}d\mu^n(M)$ over $G$, is bounded by
\begin{equation}\label{ep3}
    1 \, + \, \a\int_G s(M,z)\,d\mu^n(M) \,+\, 8\,\a^2 n^2\, \ell_{max}^2\, e^{4n\a \ell_{max}}\,.
\end{equation}
Proposition~\ref{limsupprop} asserts that for any $\e>0$ there exists $n'$ sufficiently large so that
\begin{equation}\label{ep4}
    \sup_z\int_G s(M,z)\,d\mu^n(M) \ \ \le \ \ n(\g_2\,-\,\g_1\,+\,\e)
\end{equation}
for all $n\ge n'$, and so
\begin{equation}\label{ep5}
    {\mathcal S}(n) \ \ \le \ \ 1 \,+\, n \a  (\g_2-\g_1+\e)\,+\,8\,\a^2 n^2\, \ell_{max}^2\, e^{4n\a \ell_{max}}
\end{equation}
for such $n$.  In particular, if $\e$ and $\a$ are sufficiently small, the righthand side of (\ref{ep4}) is negative  and
 ${\mathcal S}(n')<1$.  Repeated applications of the  subadditivity property (\ref{cocycle3}) show that ${\mathcal S}(kn'+m) \le {\mathcal S}(n')^k{\mathcal S}(m)$ for $1\le m \le n'$, which   implies the proposition.
\end{proof}

\subsection{Estimate on $\mu^{j-1}\d_v(B_{r,i})$}

This subsection contains the mathematical core of the argument,  a 
  H\"older estimate relating the measures $\mu^{j-1}\d_v$ and $\nu$.  For any $\e>0$ and closed subset  $U\subset \RP^{d-1}$, define a function $f=f_{\e,U}$ on $\RP^{d-1}$ by
\begin{equation}\label{holder1}
    f(x) \ \ = \ \ \max\left\{1-\f{\d(x,U)}{\e},0\right\}.
\end{equation}

\begin{proposition}\label{holdbd}
For $0<\a<1$ the function $f$ satisfies the bound
\begin{equation}\label{hbd}
    \f{|f(x)-f(y)|}{\d(x,y)^\a} \ \ \le \ \ \e^{-\a}
\end{equation}
uniformly in $x,y\in \RP^{d-1}$.
\end{proposition}

\noindent Note: the expression on the lefthand side of (\ref{holdbd}) appears in  \cite[p.~106]{lepage}, where it is use to create a Banach space norm.

\begin{proof}
The result is immediate if either $x$ and $y$ are both in $U$, or both distance at least $\e$ from $U$; likewise it is immediate if one of them  lies in $U$ and the other lies distance at least $\e$ from $U$.  We may therefore assume, without loss of generality, that $0<\d(x,U)<\e$.

If $y\in U$, the quotient equals $\e\i\d(x,U)^{1-\a}<\e^{-\a}$.
If $0<\d(y,U)<\e$,
$|f(x)-f(y)|/\d(x,y)^\a = \e\i|\d(x,U)-\d(y,U)|/\d(x,y)^{\a} \le \e\i|\d(x,U)-\d(y,U)|^{1-\a} \le \e^{-\a}$, using the inequality  
\begin{equation}\label{distanceineq}
    |\d(x,U)\,-\,\d(y,U)| \ \ \le \ \  \d(x,y)\,.
\end{equation}
For the remaining case $\d(y,U)\ge \e$ we again use (\ref{distanceineq}) to deduce $|f(x)-f(y)|/\d(x,y)^\a = \e\i|\e-\d(x,U)|/\d(x,y)^\a  \le \e\i|\e-\d(x,U)|^{1-\a}\le \e^{-\a}$.
\end{proof}

\begin{proposition}\label{holder}(See \cite[p.~107]{bougerol})
Consider the function $f$ defined in terms of the set $U$ and constant $\e>0$ in (\ref{holder1}).  For $\a$ sufficiently small, there exists $n_0>0$ and $\rho<1$ such that
\begin{equation}\label{holdprop}
    \int_G f(Mv)d\mu^n(M) \ - \ \int_{\RP^{d-1}} f(y)d\nu(y) \ \ \le \ \ \e^{-\a}\rho^n
\end{equation}
for all $n\ge n_0$.
\end{proposition}
\begin{proof}
In fact, the present argument shows this inequality holds when the lefthand side of (\ref{holdprop}) is replaced by its absolute value, though we shall not need this.
After $\nu$ by $\mu^n\ast \nu = \nu$ in the second integral,  the lefthand side equals
\begin{equation}\label{hp1}
\aligned
  & \!\!\!\!\!\!\!\!\!\!\!\!   \int_G f(Mv)d\mu^n(M) \  - \ \int_{\RP^{d-1}}\int_G f(My)d\mu^n(M)d\nu(y)
    \\
    = & \ \ \int_{\RP^{d-1}}\int_G f(Mv)d\mu^n(M)d\nu(y) \ - \ \ \int_{\RP^{d-1}}\int_G f(My)d\mu^n(M)d\nu(y)
    \\
    = & \ \ \int_{\RP^{d-1}}\int_G \(f(Mv)-f(My)\) d\mu^n(M)d\nu(y)
    \\
    \le &  \ \ \int_{\RP^{d-1}}\int_G \(\sup_{x,y} \f{|f(x)-f(y)|}{\d(x,y)^\a}  \) \d(Mv,My)^\a d\mu^n(M)d\nu(y)
    \\
    \le & \ \ \int_{\RP^{d-1}}\int_G \(\sup_{x,y} \f{|f(x)-f(y)|}{\d(x,y)^\a}  \) \(\f{\d(Mv,My)}{\d(v,y)}\)^\a d\mu^n(M)d\nu(y)\, ,
\endaligned
\end{equation}
the last inequality holding because $\d(\cdot,\cdot)\le 1$.  The result now follows from Propositions~\ref{expprop} and \ref{holdbd}.
\end{proof}

We will eventually apply this to sets containing the $B_{r,i}$ from (\ref{bri}), which are all of the form
 \begin{equation}\label{uoftheform}
\{x\in \RP^{d-1} \, \mid\, \|Ax\|<\|x\|\}
\end{equation}
for some  $A\in GL(d,\R)$ of norm greater than 1.
Given such a matrix $A$, let $w=w_A\in \R^d$ be a unit vector such that $\|Aw\|=\|A\|$.

\begin{proposition}\label{innerprod}
$\|Ax\| \ < \ \|x\| \ \ \Longrightarrow \ \ \|\langle x,w \rangle A w\| \ \le \ \|x\|\,.$
\end{proposition}
\begin{proof}
Let $z$ be a vector perpendicular to $w$.  For all $t\in \R$ we have that \begin{equation}\label{ip1}
    \|Aw\|^2 \ \ \ge \ \ \f{\|A(w+tz)\|^2}{\|w+tz\|^2} \ \  = \ \ \f{\|Aw\|^2 \, + \, 2 t \langle Aw,Az\rangle \,  + \,  t^2\|Az\|^2}{\|w\|^2   \, + \,  t^2\|z\|^2}\,,
\end{equation}
and so this last expression must have a local maximum at $t=0$.  In particular, its $t$-derivative at $t=0$  must vanish, i.e., $\langle Aw,Az\rangle =0$.  Therefore if a vector $x\in\R^d$ is decomposed as $x=\langle x,w\rangle w +z$ for some $z\perp w$, then $Ax=A\langle x,w\rangle w +Az$  is again an orthogonal decomposition.  It follows that $\|Ax\| \ge \|A\langle x,w\rangle w \|$, proving the proposition.
\end{proof}

We now return to bounding $\mu^{j-1}\d_v(B_i)$ in order to get an error estimate in (\ref{errorlastt}).
The sets $B_{r,i}$ are of the form (\ref{uoftheform}), with $A=g_r\i g_i$. We now fix  $r$ and $i$.  By Proposition~\ref{innerprod},
\begin{equation}\label{uprime}
    B_{r,i} \ \ \subset \ \ U \ \ = \ \ \left\{ x \in \RP^{d-1} \, \mid \,
        \f{|\langle x,w \rangle|}{\|x\|} \le \|A\|\i   \right\}.
\end{equation}
Proposition~\ref{holder} now shows that
\begin{equation}\label{holdsays}
    \mu^{j-1}\d_v(B_{r,i}) \ \ \le \ \   \mu^{j-1}\d_v(U)  \ \ \le \ \ \int_{\RP^{d-1}}f(y)d\nu(y) \, + \, \e^{-\a}\rho^{j-1}\, ,
\end{equation}
where $\e>0$ is arbitrary and  $f=f_{\e,U}$ is the function (\ref{holder1}). The last integral is bounded by $\nu(U')$, where
\begin{equation}\label{u2prime}
\aligned
    U' \ \ =  & \ \ \{   x \in \RP^{d-1} \,\mid\, \d(x,U)<\e         \} \\
   = & \ \ \{ x \in \RP^{d-1} \,\mid\, \exists y \text{~with~}\d(x,y)<\e\text{~and~} \f{|\langle y,w \rangle|}{\|y\|} \le \|A\|\i    \}
    \,.
   \endaligned
\end{equation}  Here $w$, as above, represents a unit vector such that $\|Aw\|=\|A\|$.  Using (\ref{dist}), this last condition on $|\langle y,w\rangle|$ can be restated as   $\d(y,w)\ge \sqrt{1-\|A\|^{-2}}$.  $U'$ is in turn contained in the set
\begin{equation}\label{u3prime}
\aligned
    U'' \ \ = & \ \ \{  x \in \RP^{d-1} \,\mid\, \d(x,w) \ge \sqrt{1-\|A\|^{-2}}-\e\} \\
    = & \ \ \{  x \in \RP^{d-1} \,\mid\, \f{\langle x,w \rangle}{\|x\|} \le \sqrt{1-(\sqrt{1-\|A\|^{-2}}-\e)^2}\}
\endaligned
\end{equation}
by the triangle inequality.

We now quote a result of Guivarc'h and Raugi (see \cite[Theorem VI.2.1]{bougerol}) which immediately implies a bound on the $\nu$-measure of $U''$ through the Chebyshev inequality.  The comments in the proof of this  Theorem on \cite[p.~156]{bougerol} indicate that the exponent $\a$ has the same source as the one in Proposition~\ref{expprop} above, and thus may be taken to have the same value.

\begin{theorem}\label{guivarch}(Guivarc'h and Raugi)
Under assumptions {\bf A1} and {\bf A2} there exists  constants $\a>0$ and $K>0$ such that
\begin{equation}\label{guivform}
    \int_{\RP^{d-1}}\left| \f{\langle x,y \rangle}{\|x\|}
      \right|^{-\a}d\nu(x)\ \ \le \ \ K
\end{equation}
uniformly in $y$.
\end{theorem}

\noindent
Applying the Chebyshev inequality to this with $y=w$, one gets
\begin{equation}\label{cheby}
\int_{\RP^{d-1}}f(y)d\nu(y) \ \ \le \ \ \nu(U'') \ \ \le \ \ K \(1-(\sqrt{1-\|A\|^{-2}}-\e)^2\)^{\a/2}\,.
\end{equation}
Therefore using (\ref{holdsays})  and assumption {\bf A3}, we bounded the $\operatorname{Prob}_{\text{Error}}$ probability from (\ref{errorlastt}) by
\begin{equation}\label{holdsays2}
\operatorname{Prob}_{\text{Error}} \ \ \le  \ \ \f 1k \sum_{\stackrel{\scriptstyle{t< j\le L}}{1\le r\neq i\le k}}   \left[ K\(1-(\sqrt{1-N^{-2}}-\e)^2\)^{\a/2}\,+\,\e^{-\a}\rho^{j-1} \right ].
\end{equation}
The expression inside the large parentheses is
$$
\f{1}{N^2}\,-\,\e^2\,+\,2\,\e\sqrt{1-\f{1}{N^2}} \ \ < \ \ \f{1}{N^2}\,+\,2\,\e\,.
$$
We now specify $\e$ to be $\f{3}{2N^2}$, so that the error is bounded by
\begin{equation}\label{holdsays3}
\operatorname{Prob}_{\text{Error}} \ \ \le \ \ \f 1k \sum_{\stackrel{\scriptstyle{t< j\le L}}{1\le r\neq i\le k}}   \left( K2^{\a}N^{-\a}\,+\,\e^{-\a}\rho^{j-1} \right).
\end{equation}
Take $t=\lceil 1+\f{\log{(3^\a K N^{-3\a})}}{\log \rho} \rceil$, so that
\begin{equation}\label{domination}
    K2^\a N^{-\a} \ \ > \ \ \e^{-\a}\rho^{j-1} \ \ \ \ \text{for~}j \ \ge \ t
\end{equation}
and
\begin{equation}\label{holdsays4}
    \operatorname{Prob}_{\text{Error}} \ \ \le \ \ \f{1}{k}\cdot(L-t)k(k-1)\cdot 2^{\a+1} K N^{-\a}\,.
\end{equation}
This completes the proof of Theorem~\ref{positiveresult}.

\subsection{Numerical Examples}\label{numericalexample}

\subsubsection*{Example 1:~where Norm Reduction works well}

We now present an example of the algorithm in practice, for dimension $d=3$ and the generating set $S=\{g_1,g_2,g_3\}$, where
\begin{equation}\label{explicS}
    g_1 \ = \ \(\begin{smallmatrix}
                  -9 & -59 & 30 \\
                  11 & 66 & -32 \\
                  3 & 21 & -11 \\
                \end{smallmatrix}
    \) \  , \ \ \
     g_2 \ = \ \(\begin{smallmatrix}
                  444& -31 & -363 \\
                  -110 & 7 & 90 \\
                  -1271 & 90 & 1039 \\
                \end{smallmatrix}
    \)
     \  , \ \ \ \text{and} \ \
      g_3 \ = \ \(\begin{smallmatrix}
                  9 & 31 & 33 \\
                  -91 & -303 & -310 \\
                  -35 & -116 & -118 \\
                \end{smallmatrix}
    \)\!.
\end{equation}
\begin{table}[h]
  \centering
\begin{tabular}{|c|c|c|}
 \hline
 $L$ & Number of Attempts & Number of Successes \\
  \hline
  2 & 10,000 & 10,000 \\
  10 & 10,000 & 9,998 \\
  50 & 10,000 & 9,978 \\
  100 & 10,000 & 9,963 \\
  200 & 10,000 & 9,936 \\
  1,000 & 1,000 & 1,000 \\
  \hline
\end{tabular}
  \caption{Numerical results with generating set $S$ from (\ref{explicS}).}\label{resultstable}
\end{table}
\noindent
These matrices were chosen randomly among those with integral entries in  a bounded range. In all our tests we ran the algorithm with the parameter $t=0$, i.e., not allowing for brute force search for the final steps.
The parameter $N$ in this example is $\approx\!12157.1$.
We ran several numerical trials of the Norm Reduction Algorithm (\ref{NRA}) on the Word Problem on Vectors (\ref{WordProblemonVectors})  with the vector $v=(1,0,0)$, almost all of which were successful (see Table~\ref{resultstable}).

The error term (\ref{holdsays4}) is bounded by the one given in Theorem~\ref{positiveresult} if $C$ is taken to be $4K$.
In this typical example, the invariant measure $\nu$  and its approximations $\mu^n\ast \d_v$ are supported near the eigenvectors for the $g_i$ corresponding to their maximal eigenvalue.  Recall that the constant $K$ comes from the measure of the set $U''$, which in (\ref{u3prime}) is related to points in $\RP^2$ which have $\d$-distance very close to 1 from the direction of maximal stretching of the six matrices $g_r\i g_i$.
We computed that these 18 pairs of $\d$-distances range between .33 and .98, far from 1 on the scale of $1/N$.  Since $C$ can be large only if these distances are much closer to 1, we concluded that $C$ is small -- under some heuristics, we computed its value to be below 7.

To estimate the value of $\a$, we recall its origin in Proposition~\ref{expprop} comes from bounds on the quantities ${\cal S}(n)$ (\ref{cocycle2}).  We numerically estimated that  ${\cal S}(2)<.83$ for $\a=0.4$.  This was done by approximating that maximum using a mesh.  While that is no guarantee of an accurate estimate for the maximum, it is worth noting that the values to be maximized were typically much smaller.  Also,  using ${\cal S}(n)$ for larger values of $n$ would result in a better estimate for $\a$.
With this value of $\a=.4$, the probability in Theorem~\ref{positiveresult} is  less than 1 only for small values  of $L-t$.  However, that estimate is certainly an overestimate for other reasons:~for one thing, the proof   estimates the error probability at each step, and  multiplies this individual estimate by the number of steps to obtain the final estimate.  The actual error probability is likely to be far smaller.  The combination of this potential  to improve the estimates, along with the excellent performance of the Norm Reduction Algorithm (\ref{NRA}) in practice, demonstrates its usefulness in attacking the Word Problem on Vectors (\ref{WordProblemonVectors}).

\subsubsection*{Example 2:~where Norm Reduction does not work well}

The algorithm does not perform well when one of the generators is orthogonal.  
In this example we take $S'=\{g_1',g_2,g_3\}$, where $g_1'=\(\begin{smallmatrix}
                  0 & 0 & 1 \\
                  0 & 1 & 0 \\
                  1 & 0 & 0 \\
                \end{smallmatrix}
    \)$ and $g_2$, $g_3$ are as defined in (\ref{explicS}).  With this one change (but otherwise the same conditions as in Example 1) the outcomes were much worse, and are summarized in Table~\ref{badtable}.
    \begin{table}[h] 
  \centering
\begin{tabular}{|c|c|c|}
 \hline
 $L$ & Number of Attempts & Number of Successes \\
  \hline
  2 & 10,000 & 10,000 \\
  10 & 10,000 & 4,404 \\
  50 & 10,000 & 86 \\
  100 & 10,000 & 2 \\
  200 & 10,000 & 0 \\
  1000 & 1000 & 0 \\
  \hline
\end{tabular}
  \caption{Numerical results with generating set $S'$.}\label{badtable}
\end{table}

\section{Rounding and the Traveling Salesman Problem}\label{tspsec}

In this section we show how algorithms to solve the Closest Group Element Problem (\ref{CGEP}) can be easily converted to solve the Traveling Salesman Problem (\textsc{tsp}), and in particular prove Theorem~\ref{negthm}.

\begin{definition}\label{tsp} {\em Traveling Salesman Problem (on graphs)}.  Given a complete graph on $n$ vertices whose edges have positive integer weights, find a Hamiltonian cycle which has minimal total weight (i.e., sum of its edge weights).
\end{definition}

\noindent The above formulation is  more general than the metric \textsc{tsp} problem, in that the edge weights  do not need to obey the triangle inequality.  The \textsc{tsp} problem is NP-hard, as is the simpler problem of finding a Hamiltonian cycle whose total weight is within a constant factor of the minimum \cite[Theorem~3.6]{vazirani}.

We shall now describe how to convert any instance of \textsc{tsp}  into a Closest Group Element Problem (\ref{CGEP}).  First we set some notation for the \textsc{tsp} problem.  Let $w_e=w_{ij}=w_{ji}$ be the weight of the directed edge $e=(i,j)$ connecting  the $i$-th and $j$-th vertices. Let $m$ be an {\em a priori} lower bound for the total weight of the shortest Hamiltonian cycle (for example, $m$ can be $n$ times the lowest edge weight), and $M$ be an upper bound (for example, the weight of any Hamiltonian cycle).  Let $m_0$ denote the minimal total weight, which is unknown (and hence which we do not use in setting parameters).  Since the weights are positive integers, one may of course assume that $m_0,m\ge 1$.
The edge weight unit can be rescaled without affecting the solution to the \textsc{tsp} problem:~accordingly we shall replace the above parameters by $mT,MT$, and $m_0T$,where $T>0$ is a parameter that will be chosen later.  After this rescaling, one has that
\begin{equation}\label{isolatedmin}
    \text{any cycle  weight less than $m_0T+T$ is minimal.}
\end{equation}
In particular, there is no loss of generality in assuming that $M\ge m_0+1$.
Given an edge $e=(i,j)$, let $v_e$ denote the row vector of length $n$ which has all zeroes except for $K$'s in positions $i$ and $j$, where $K$ is a parameter that will be chosen later.  Let $E_{ij}$ denote the $n\times n$ matrix which has all 0 entries except a 1 in the $(i,j)$-th position.  Let $\b>\a\ge 0$ be parameters (to be specified later), and $M_e=M_{ij}=\a I+\b E_{ij}$.  We set $d=2n+3$ and define $d\times d$ matrices for each directed edge by
\begin{equation}\label{gedge}
    g_e \ \  = \ \ g_{ij} \ \ = \ \ \(\begin{smallmatrix}
    M_e & & & & \\
    & 1 & v_e & & \\
    & & I_{n\times n} & & \\
    & & & 1 & w_e \\
    & & & & 1
    \end{smallmatrix}\)
\end{equation}
(the blocks in this matrix are of sizes $n$, $1$, $n$, $1$, and $1$, respectively; we have as well used the convention that blank entries are zero).
Note that $E_{ij}\neq E_{ji}$, and consequently $M_{ij}\neq M_{ji}$ and $g_{ij}\neq g_{ji}$.  The Zariski closure of the group (or semigroup) generated by $\{g_{i,j}|i\neq j\}$ contains $GL_n(\R)$, embedded into the $n\times n$ block in the upper left corner, and satisfies the large dimensionality constraint of Section 1.

If $h_1,\ldots,h_\ell$ are all square matrices of the same size, let $\prod_{i\le \ell}h_i$ denote the product $h_1\cdots h_\ell$.
If $e_1,\ldots,e_\ell$ are edges, then
\begin{equation}\label{multlaw}
    g_{e_1}g_{e_2}\cdots g_{e_\ell} \ \ = \ \ \(\begin{smallmatrix}
    \prod\!M_{e_r} & & & & \\
    & 1 & \sum\!v_{e_r} & & \\
    & & I_{n\times n} & & \\
    & & & 1 & \sum\!w_{e_r} \\
    & & & & 1
    \end{smallmatrix}\).
\end{equation}
We shall now see how features of this matrix are related to the total weights of Hamiltonian cycles.  First of all, $\sum_{r\le \ell} v_{e_r}$ equals $[2K\,\ldots\,2K]$ (i.e., a vector of all $2K$'s) if and only if the edges $e_1,\ldots,e_\ell$ touch each vertex exactly twice.  The entry $\sum_{r\le \ell} w_{e_r}$ is of course the total weight of the path, if indeed $e_1,\ldots,e_r$ trace out a path.  The product
\begin{equation}\label{Mer}
    \prod_{r\le \ell}M_{e_r} \ \ = \ \ \prod_{r\le \ell}\(\a I \,+\,\b E_{i_rj_r}\) \ \   = \ \ \sum_{(\e_1,\ldots,\e_\ell)\in\{0,1\}^\ell}
    \a^{\ell-(\e_1+\cdots+\e_\ell)}\b^{\e_1+\cdots+\e_\ell} \prod_{r\le \ell} E_{i_rj_r}^{\ \ \ \e_r}
\end{equation}
helps detect such a path.  The last product is zero unless the edges $e_r$ for which $\e_r=1$ trace out a connected path; if they do, the product equals $E_{ij}$, where $i$ is the first value of $i_r$ for which $\e_r=1$ and $j$ is  the last value of $j_r$ for which $e_r=1$.  Note that if  $\a=0$, the only nonzero term is the one for $\e_1=\e_2=\cdots=\e_\ell=1$: then the product $\prod_{r\le \ell} M_{e_r}=\b^\ell E_{i_1j_\ell}$ if the edges $e_1,e_2,\ldots,e_\ell$ trace out a connected path, but is zero otherwise.  Thus in the extreme case $\a=0$, tracing out a connected path is equivalent to the nonvanishing of this product.  Unfortunately, however,  the matrices are only invertible if $\a>0$.  We will mainly be concerned with the case of $\a>0$ because of its relevance to the Closest Group Element Problem (\ref{CGEP}), but include some comments about the $\a=0$ case as well.  In fact, the extra parameters $\a$ and $\beta$ are needed simply to adapt features of the simpler $\a=0$ case to noninvertible matrices.

\begin{proposition}\label{crudedetect}The $(i,j)$-th entry of $\prod_{r\le \ell} M_{e_r}$ satisfies the bound
\begin{equation}\label{cd1}
 \(\prod_{r\le \ell} M_{e_r}\)_{ij}       \ \ \le \ \ \a\,\b^{\ell-1}\,2^\ell
\end{equation}
if the edges $e_1,e_2,\ldots,e_\ell$ do not trace out a path,
and
\begin{equation}\label{cd2}
    \(\prod_{r\le \ell} M_{e_r}\)_{ij} \ -\ \b^\ell\,\d_{i=i_1}\,\d_{j=j_\ell}           \ \ \le \ \ \a\,\b^{\ell-1}\,2^\ell
\end{equation}
if they do.
\end{proposition}
\begin{proof}
In either case, the expressions to be bounded are the matrix entries of the sum on the righthand side of (\ref{Mer}), except for the term corresponding to $\e_1=\e_2=\cdots=\e_\ell=1$ (which only comes up in (\ref{cd2}) anyhow).  The matrix entries of a product of $E_{i_rj_r}$ are all $\le 1$, so the sum is bounded by $(\a+\b)^\ell-\b^\ell \le \a \b^{\ell-1}2^\ell$.
\end{proof}

\noindent
Let $\e>0$ be a parameter (which will be specified later).
The Closest Group Element Problem (\ref{CGEP}) derived from this \textsc{tsp} instance is the following, assuming $\a>0$ (if $\a=0$, it is the verbatim rounding problem for {\em semigroups}):
\begin{multline}\label{roundtoz}
   \text{Find the closest product of length $\le n$ of the $g_e$'s  to the matrix~} \\  z \ \ = \ \ \(\begin{smallmatrix}
    \b^n E_{11}+\e I_n & & & & \\
    & 1 &  2K\cdots2K & & \\
    & & I_{n\times n} & & \\
    & & & 1 & 0 \\
    & & & & 1
    \end{smallmatrix}\) \ \ \text{in the matrix norm~(\ref{l2matnorm}).}
\end{multline}
The block structure of the matrices allows us to compute the distance of a product in terms of the features described after (\ref{multlaw}):
\begin{multline}\label{tspstar}
    \left\| \, \prod_{r\le \ell}g_{e_r} \,-\,z\,\right\|^2 \ \   = \ \
\left\| \, \prod_{r\le \ell}M_{e_r} \,-\,\b^nE_{11}\,-\,\e\, I_n\,   \right\|^2 \ + \\
\left\| \, \sum_{r\le \ell}v_{e_r} \,-\,[2K\cdots 2K]\,   \right\|^2
\ + \ \(\sum_{r\le\ell}w_{e_r}\)^2\!\!,
\end{multline}
where we again stress that $\|\cdot\|$ refers to the norm (\ref{l2matnorm}) for the rest of this section.

\begin{proposition}\label{abcprop} (Note that $\ell=n$ in parts (A) and (C).)

(A) If the edges $e_1,e_2,\ldots,e_n$ trace out a Hamiltonian cycle starting and ending at the first vertex, then
\begin{equation}\label{abcaa}
    \left\| \, \prod_{r\le \ell}M_{e_r} \,-\,\b^nE_{11}\,-\,\e\,I_n\,   \right\|^2 \ \ \le \ \ (n  \, \a \,  \b^{n-1} \, 2^n)^2 \ + \ (n\,\e)^2
\end{equation}
and consequently
\begin{equation}\label{abca}
     \left\| \, \prod_{r\le n}g_{e_r} \,-\,z\,\right\|^2 \ \ \le \ \ (n\, \a \, \b^{n-1} \, 2^n)^2 \ + \ (n\,\e)^2 \ + \ \ (\text{weight of cycle})^2.
\end{equation}
(B) If the edges $e_1$, $e_2$,\ldots, $e_\ell$ do not touch each vertex exactly twice, then
\begin{equation}\label{abcb}
  \left\| \, \prod_{r\le \ell}g_{e_r} \,-\,z\,\right\|^2 \ \ \ge \ \ K^2.
\end{equation}
(C) If the  edges $e_1$, $e_2$,\ldots, $e_n$ do not trace out a path beginning and ending at vertex 1, then
\begin{multline}\label{abcc}
\left\| \, \prod_{r\le n}M_{e_r} \,-\,\b^nE_{11}\,-\,\e\,I_n\,   \right\|^2 \ \
\ge \ \  \left| \, \(\prod_{r\le n}M_{e_r}\)_{11} \,-\,\b^n \,-\,\e   \,   \right|^2 \\ \ge \ \ (\e+\b^n-\a\b^{n-1}2^n)^2.
\end{multline}
\end{proposition}
\begin{proof}
The inequality (\ref{abcaa}) in part (A) is an immediate consequence of (\ref{cd2}) and the triangle inequality.  It then implies (\ref{abca}) because the middle term on the righthand side of (\ref{tspstar}) vanishes when the path enters and exists each vertex exactly once.

On the other hand, failure to touch each vertex exactly twice means one of the vector entries for the middle term in (\ref{tspstar}) will be at least $K$, showing that the righthand side of (\ref{abcb}) is at least $K^2$ (in fact by parity considerations it will be at least $2K^2$).   This demonstrates part (B).
 Part (C) is likewise a consequence of Proposition~\ref{crudedetect}.
\end{proof}

\begin{proposition}\label{123prop}
Suppose
\begin{enumerate}
  \item $(n \a  \b^{n-1}2^n)^2 \,+\,(n\e)^2 \ < \ m T^2$
  \item $K \ \ge  \ MT$
  \item $\e+\b^n\,-\,\a\b^{n-1}2^n \ \ge \  MT$.
\end{enumerate}
Then   any word of length $\le n$ in the $g_e$ closest to $z$ has the form $g_{e_1}g_{e_2}\cdots g_{e_n}$, where the edges $e_1,e_2,\ldots,e_n$ trace out a  Hamiltonian cycle of shortest total weight that begins and ends at the first vertex.
\end{proposition}
\begin{proof}
We shall use all three parts of the previous Proposition.  Part (A) and property 1 imply that if $e_1,e_2,\ldots,e_n$ is the shortest Hamiltonian cycle and $h_1=g_{e_1}g_{e_2}\cdots g_{e_n}$, then (\ref{tspstar}) implies 
\begin{multline}\label{a1}
     \left\| \, h_1 \,-\,z\,\right\|^2 \ \  \le \ \ (n \a  \b^{n-1}2^n)^2\,+\,(n\e)^2\,+\,(m_0T)^2 \\
 < \ \ m T^2 \, + \, m_0^2T^2 \ \
 \le \ \ T^2(m_0+1)^2\,,
\end{multline}
because $m\le m_0$.

Part (B) and property 2 imply that a path which does not touch each vertex exactly twice has
\begin{equation}\label{b2}
    M^2T^2 \ \ \le \ \  \left\| \, \prod_{r\le \ell}g_{e_r} \,-\,z\,\right\|^2.
\end{equation}
Since we have assumed $M \ge m_0+1$, the word $\prod_{r\le \ell}g_{e_r}$ cannot be closest to $z$. In particular, the closest word to $z$ must be a product of length exactly $n$ (otherwise the edges it is formed from do not touch each vertex exactly twice).
Part (C) and property 3 likewise show that the edges of the closest word trace out a path beginning and ending at 1.

Thus the closest word comes from a Hamiltonian cycle.  We now must show that it comes from the  Hamiltonian cycle of lowest total weight.  Indeed, suppose that $h=\prod_{r\le n}g_{e_r}$ comes from a Hamiltonian cycle and
 $\|h-z\|< \|h_1-z\|$.  By (\ref{tspstar}) and (\ref{a1}) we must have
\begin{equation}\label{hdistz}
(\text{total weight of $h$'s path})^2 \ \  \le \ \ \|h-z\|^2 \ \
 < \ \ \|h_1-z\|^2 \ \
\le \ \  T^2(m_0+1)^2\,,
\end{equation}
and property (\ref{isolatedmin}) shows that this path is minimal -- a contradiction.
\end{proof}

\begin{proposition}\label{approxfactorreduct}  Suppose edges $e_1,e_2,\ldots,e_n$ trace out a Hamiltonian cycle starting and ending at the first vertex, and whose total weight is $\le m_0 T A$ for some $A\ge 1$ (that is,  within a factor $A$ of being minimal).  Suppose furthermore that
\begin{equation}\label{assumeapproxfactorreduct}
(n\a\b^{n-1}2^n)^2 \, + \, (n\e)^2 \ \ \le \ \  (m A T)^2\,,
\end{equation}
which is a consequence of the first assumption of Proposition~\ref{123prop} since $m,A\ge 1$.
  If $h=g_{e_1}g_{e_2}\cdots g_{e_n}$ (respectively,  $h'$) is the word formed from this cycle (respectively, a minimal cycle), then
\begin{equation}\label{approxreduct}
    \|h-z\| \ \ \le \ \ \sqrt{2}\,A\,\|h'-z\|.
\end{equation}
\end{proposition}
\begin{proof}
By (\ref{abca}) one has
\begin{multline}\label{approxfactorreduct1}
    \|h-z\|^2 \ \ \le \ \ (n\a\b^{n-1}2^n)^2\,+\,(n\e)^2\,+\,(\text{weight of path})^2 \ \ \le \\ m^2A^2T^2\,+\,m_0^2A^2T^2 \ \ \le \ \ 2 (m_0AT)^2.
\end{multline}
The result follows because  $\|h'-z\|\ge m_0T$ by (\ref{tspstar}).
\end{proof}

The conditions of the previous Propositions can be achieved with matrix entries that are polynomially-sized in the input of the \textsc{tsp} instance.  For example, the following parameter choices are easily checked to satisfy them.

\begin{proposition}\label{examplesof123}
Properties 1, 2, and 3 of Proposition~\ref{123prop} as well as (\ref{assumeapproxfactorreduct}) hold under the following parameter choices.

(i)  $\a=1$, $\b=\max(2^{n+4},4M^2,\f{n^22^{2n+1}}{m})$, $\epsilon=\f{1}{2n}$, $T=\b^{n-1/2}$, and $K=MT$.  In this case the matrices $g_e$ all have determinant 1.

(ii) $T=1$, $K=M$, $\b=(2M)^{1/n}$, $\a=\f{\sqrt{m/2}}{n2^n\b^{n-1}}$,  and $\e=\f{\sqrt{m/2}}{n}$.  In this case the matrices all have determinant $\a^n$ (and are hence invertible).

(iii) $\a=0$,  $\b=M^{1/n}$, $\e=0$, $K=M$, and  $T=1$.  In this case the matrices are not invertible.
\end{proposition}

\noindent Since the entries in these matrices  $g_{ij}$ and $z$ are polynomially sized, Theorem~\ref{negthm} then follows immediately from Proposition~\ref{approxfactorreduct} and the corresponding inapproximability of the Traveling Salesman Problem on graphs \cite[Theorem~3.6]{vazirani}.


\begin{bibsection}


\begin{biblist}
                                                                                            %
\bib{arora93hardness}{article}{
author={Arora, Sanjeev},
author={Babai, L{\'a}szl{\'o}},
author={Stern, Jacques},
author={Sweedyk, Z.},
title={The hardness of approximate optima in lattices, codes, and systems
of linear equations},
note={34th Annual Symposium on Foundations of Computer Science (Palo
Alto, CA, 1993)},
journal={J. Comput. System Sci.},
volume={54},
date={1997},
number={2},
pages={317--331}
}

\bib{babai}{article}{   author={Babai, L.},
   title={On Lov\'asz' lattice reduction and the nearest lattice point
   problem},
   journal={Combinatorica},
   volume={6},
   date={1986},
   number={1},
   pages={1--13}
}

\bib{bougerol}{book}{
author={Bougerol, Philippe},
author={Lacroix, Jean},
title={Products of random matrices with applications to Schr\"odinger
operators},
series={Progress in Probability and Statistics},
volume={8},
publisher={Birkh\"auser Boston Inc.},
place={Boston, MA},
date={1985}
}

\bib{dinur98approximatingcvp}{article}{
author={Dinur, I.},
author={Kindler, G.},
author={Safra, S.},
title={Approximating CVP to within almost-polynomial factors is NP-hard},
conference={
title={Symposium on Foundations of Computer Science}},
book={
publisher={IEEE},
place={New York},
},
pages={99-111},
date={1998}
}

\bib{fuchs}{article}{
   author={Fuchs, Elena},
   title={The ubiquity of thin groups},
   conference={
      title={Thin groups and superstrong approximation},
   },
   book={
      series={Math. Sci. Res. Inst. Publ.},
      volume={61},
      publisher={Cambridge Univ. Press, Cambridge},
   },
   date={2014},
   pages={73--92},
   review={\MR{3220885}},
}
	

\bib{furst}{article}{
   author={Furstenberg, Harry},
   title={Noncommuting random products},
   journal={Trans. Amer. Math. Soc.},
   volume={108},
   date={1963},
   pages={377--428}
}

\bib{FK}{article}{
   author={Furstenberg, H.},
   author={Kesten, H.},
   title={Products of random matrices},
   journal={Ann. Math. Statist},
   volume={31},
   date={1960},
   pages={457--469}
}


\bib{BraidEqns}{article}{
author={Garber, D.}, author={Kaplan, S.}, author= {Teicher, M.},author={Tsaban, B.}, author = {Vishne, U.},title={Probabilistic solutions of equations in the braid
group},journal={Advances in Applied
Mathematics},volume={35},year={2005},pages={323--334},
note={\url{http://arxiv.org/abs/math/0404076}}}


\bib{Gurevich}{article}{
   author={Gurevich, Yuri},
   author={Schupp, Paul},
   title={Membership problem for the modular group},
   journal={SIAM J. Comput.},
   volume={37},
   date={2007},
   number={2},
   pages={425--459 (electronic)}
}
		

\bib{gromov}{article}{
   author={Gromov, M.},
   title={Asymptotic invariants of infinite groups},
   conference={
      title={Geometric group theory, Vol.\ 2},
      address={Sussex},
      date={1991},
   },
   book={
      series={London Math. Soc. Lecture Note Ser.},
      volume={182},
      publisher={Cambridge Univ. Press},
      place={Cambridge},
   },
   date={1993},
   pages={1--295}
}


\bib{joux98lattice}{article}{
author={Joux, Antoine},
author={Stern, Jacques},
title={Lattice reduction: a toolbox for the cryptanalyst},
journal={J. Cryptology},
volume={11},
date={1998},
number={3},
pages={161--185}
}

\bib{lll}{article}{
author={Lenstra, A. K.},
author={Lenstra, H. W., Jr.},
author={Lov{\'a}sz, L.},
title={Factoring polynomials with rational coefficients},
journal={Math. Ann.},
volume={261},
date={1982},
number={4},
pages={515--534}
}

\bib{lepage}{article}{
   author={Le Page, {\'E}mile},
   title={Th\'eor\`emes limites pour les produits de matrices al\'eatoires},
   language={French},
   conference={
      title={Probability measures on groups},
      address={Oberwolfach},
      date={1981},
   },
   book={
      series={Lecture Notes in Math.},
      volume={928},
      publisher={Springer},
      place={Berlin},
   },
   date={1982},
   pages={258--303}
}

\bib{LMR}{article}{
   author={Lubotzky, Alexander},
   author={Mozes, Shahar},
   author={Raghunathan, M. S.},
   title={The word and Riemannian metrics on lattices of semisimple groups},
   journal={Inst. Hautes \'Etudes Sci. Publ. Math.},
   number={91},
   date={2000},
   pages={5--53 (2001)}
}

\bib{margulis}{book}{
   author={Margulis, G. A.},
   title={Discrete subgroups of semisimple Lie groups},
   series={Ergebnisse der Mathematik und ihrer Grenzgebiete (3) [Results in
   Mathematics and Related Areas (3)]},
   volume={17},
   publisher={Springer-Verlag},
   place={Berlin},
   date={1991},
   pages={x+388}
}


\bib{markov}{book}{author={Markov, A. A.},title={The theory of Algorithms},year={1961},series={Israel Program for Scientific Translation},place={ Jerusalem}}

\bib{mihail}{article}{
author={Mihailova, K. A.},
title={The occurrence problem for direct products of groups},
language={Russian},
journal={Mat. Sb. (N.S.)},
volume={70 (112)},
date={1966},
pages={241--251}
}

\bib{rudin}{book}{
author={Rudin, Walter},
title={Functional analysis},
series={International Series in Pure and Applied Mathematics},
edition={2},
publisher={McGraw-Hill Inc.},
place={New York},
date={1991}
}

\bib{Thomp1}{article}{
author={ Ruinskiy, D.}, author={Shamir, A.}, author={Tsaban,B.},title={Length-based
cryptanalysis:~The case of Thompson's Group},journal={Journal of Mathematical
Cryptology},volume={1}, year={2007},pages={359--372}}

\bib{sarnak}{article}{
   author={Sarnak, Peter},
   title={Notes on thin matrix groups},
   conference={
      title={Thin groups and superstrong approximation},
   },
   book={
      series={Math. Sci. Res. Inst. Publ.},
      volume={61},
      publisher={Cambridge Univ. Press, Cambridge},
   },
   date={2014},
   pages={343--362},
   review={\MR{3220897}},
}


\bib{shamir82}{article}{
author={Shamir, Adi},
title={A polynomial time algorithm for breaking the basic Merkle-Hellman
cryptosystem},
conference={title={Crypto}},
date={1982},
pages={279--288}
}

\bib{logic}{article}{
author = {Tsitsiklis, J.},
author={Blondel, V.},
  title = {The Lyapunov exponent and joint spectral radius of pairs of matrices are
    hard - when not impossible - to compute and to approximate},
journal={Mathematics of Control, Signals, and Systems},
volume={    10},
pages={31--40},
note={Correction in {\bf 10}, p. 381},
  year = {1997}}

\bib{vazirani}{book}{  author={Vazirani, Vijay V.},
   title={Approximation algorithms},
   publisher={Springer-Verlag, Berlin},
   date={2001},
   pages={xx+378},
   isbn={3-540-65367-8},
   review={\MR{1851303 (2002h:68001)}},
}


\bib{zimmer}{book}{
   author={Zimmer, Robert J.},
   title={Ergodic theory and semisimple groups},
   series={Monographs in Mathematics},
   volume={81},
   publisher={Birkh\"auser Verlag},
   place={Basel},
   date={1984},
   pages={x+209}
}

\end{biblist}
\end{bibsection}
\end{document}